\documentclass[12pt]{elsarticle}
\usepackage{float}
\usepackage{amssymb}
\usepackage{amsthm}
\usepackage{enumerate}
\usepackage{amsmath,amssymb,amsfonts,amsthm,fancyhdr}
\usepackage{mathrsfs}
\usepackage[colorlinks,linkcolor=black,anchorcolor=black,citecolor=black,hyperindex=black,CJKbookmarks=black]{hyperref}
\usepackage{tikz}
\usepackage{mathabx}

\newtheorem{theorem}{Theorem}[section]
\newtheorem{lemma}[theorem]{Lemma}
\newdefinition{remark}[theorem]{Remark}
\newdefinition{corollary}[theorem]{Corollary}
\newdefinition{prop}[theorem]{Proposition}
\numberwithin{equation}{section}
\begin{document}
\begin{frontmatter}
\title{Dimensions of certain sets of continued fractions with non-decreasing partial quotients}
\author[a]{Lulu Fang}\ead{fanglulu1230@gmail.com}
\author[b]{Jihua Ma}\ead{jhma@whu.edu.cn}
\author[c]{Kunkun Song\corref{cor1}}\ead{songkunkun@hunnu.edu.cn}
\author[d]{Min Wu}\ead{minwu@scut.edu.cn}
\address[a]{School of Mathematics and Statistics, Nanjing University of Science and Technology, Nanjing, 210094, China}
\address[b]{School of Mathematics and Statistics, Wuhan University, Wuhan, 430072, China}
\address[c]{Key Laboratory of Computing and Stochastic Mathematics (Ministry of Education), Key Laboratory of Control and Optimization of Complex Systems (College of Hunan Province), School of Mathematics and Statistics, Hunan Normal University, Changsha, 410081, China}
\address[d]{School of Mathematics, South China University of Technology, Guangzhou, 510640, China}
\cortext[cor1]{Corresponding author.}

\begin{abstract}\par
Let $[a_1(x),a_2(x),a_3(x),\cdots]$ be the continued fraction expansion of $x\in (0,1)$.
This paper is concerned with certain sets of continued fractions with non-decreasing partial quotients.
As a main result, we obtain the Hausdorff dimension of the set
\[\left\{x\in(0,1): a_1(x)\leq a_2(x)\leq \cdots,\ \limsup\limits_{n\to\infty}\frac{\log a_n(x)}{\psi(n)}=1\right\}\]
for any  $\psi:\mathbb{N}\rightarrow\mathbb{R}^+$ satisfying $\psi(n)\to\infty$ as $n\to\infty$.
\end{abstract}


\end{frontmatter}

\section{Introduction}
Every real number $x \in (0,1)$ admits a \emph{continued fraction expansion} of the form
\begin{equation}\label{cfe}
x = \dfrac{1}{a_1(x) +\dfrac{1}{a_2(x)  +\dfrac{1}{a_3(x)+\ddots}}}:=[a_1(x), a_2(x), a_3(x),\cdots],
\end{equation}
where the \emph{partial quotients} $a_1(x),a_2(x),a_3(x),\cdots$ are positive integers.
 Basic properties of continued fractions may be found in \cite{IK02, Khi64} and references therein.\\
\indent
This paper falls into the category of the metric theory of continued fractions. We begin with the Borel-Bernstein theorem (see \cite{BF1912,BE1909,BE1912}), which states that for any $\psi:\mathbb{N}\rightarrow\mathbb{R}^+$ the set
 $$B(\psi)=\Bigl\{x \in (0,1):a_n(x) \geq \psi(n)\;\text{for infinitely many\;$n$'s}\Bigr\}$$
 has Lebesgue measure either $0$ or $1$ according as the series $\sum_{n \geq 1} 1/\psi(n)$ converges or diverges. We are interested at the Hausdorff dimension of the set $B(\psi)$ in the first case when it is of Lebesgue measure $0$.
 As a first result toward this direction, Good \cite{Good41} obtained some estimations on the Hausdorff dimension of $B(\psi)$ in 1941. A complete solution to the problem about the Hausdorff dimension of $B(\psi)$ was given by Wang and Wu \cite{WW08}. In other directions, there are many papers investigating the Hausdorff dimension of sets of continued fractions with some restrictions on the growth rate of their partial quotients, see for example, Hirst \cite{Hirst}, Cusick \cite{Cus90}, Wang and Wu \cite{lesWW08A}, Cao, Wang and Wu \cite{CWW}, Takahasi \cite{HT21}.\\
\indent As a consequence of the Borel-Bernstein theorem, for Lebesgue almost all $x \in (0,1)$,
\begin{equation}\label{limsupan}
\limsup\limits_{n\to\infty}\frac{\log a_n(x)}{\log n}=1.
\end{equation}
It is shown in \cite{FMS20} that the set of points for which the limsup in \eqref{limsupan} equals to a given non-negative real number has full Hausdorff dimension. More generally, let $\psi:\mathbb{N}\rightarrow\mathbb{R}^+$ be a function satisfying $\psi(n)\to\infty$ as $n\to\infty$. Fang, Ma and Song \cite{FMS20} calculated the Hausdorff dimensions of the sets
\[
E_{\sup}(\psi)=\Bigl\{x\in(0,1):\ \limsup\limits_{n\to\infty}\frac{\log a_n(x)}{\psi(n)}=1\Bigr\},
\]
\[
E_{\inf}(\psi)=\Bigl\{x\in(0,1):\ \liminf\limits_{n\to\infty}\frac{\log a_n(x)}{\psi(n)}=1\Bigr\}
\quad \text{and}\quad E(\psi)=E_{\sup}(\psi)\cap E_{\inf}(\psi).
\]
Throughout this paper,  we use $\dim_{\rm H}$ to denote the Hausdorff dimension (see \cite{Fal90}).
For the reader's convenience, we list the main results in \cite{FMS20} as follows.
\begin{theorem}\label{RAMA}
Let $\psi:\mathbb{N}\rightarrow\mathbb{R}^+$ be a function satisfying $\psi(n)\to\infty$ as $n\to\infty$.
\begin{enumerate}[(i)]
\item If $\psi(n)/n\to0$ as $n\to\infty$, then $\dim_{\rm H}E_{\sup}(\psi)=1$.
\item If $\psi(n)/n\to\alpha\ (0<\alpha<\infty)$ as $n\to\infty$, then $\dim_{\rm H}E_{\sup}(\psi)=S(\alpha)$,
where $S:\mathbb{R}^+\rightarrow(1/2,1)$ is a continuous function satisfying
\[\lim\limits_{\alpha\to0}S(\alpha)=1\ \ \text{and}\ \ \lim\limits_{\alpha\to\infty}S(\alpha)=\frac{1}{2}.\]
\item If $\psi(n)/n\to\infty$ as $n\to\infty$, then $\dim_{\rm H}E_{\sup}(\psi)=1/(A+1)$,
where $A \in [1,\infty]$ is given by
\begin{equation}\label{Asup}
\log A:=\liminf\limits_{n\to\infty}\frac{\log\psi(n)}{n}.
\end{equation}
\item $\dim_{\rm H}E_{\inf}(\psi)=1/(B+1)$,
where $B \in [1,\infty]$ is given by
\begin{equation}\label{Binf}
\log B:=\limsup\limits_{n\to\infty}\frac{\log\psi(n)}{n}.
\end{equation}
\item $\dim_{\rm H}E(\psi)=1/(C+1)$,
where $C \in [1,\infty]$ is given by
\begin{equation}\label{Clim}
C:=1+\limsup\limits_{n\to\infty}\frac{\psi(n+1)}{\psi(1)+\cdots+\psi(n)}.
\end{equation}
\end{enumerate}
\end{theorem}
\noindent Moreover, they also remarked that $A\leq B\leq C$ and these three values of $A, B$ and $C$ can be all different for some functions $\psi$.

Recently, the authors of \cite{FMSW21} studied the Hausdorff dimension of the intersection of $E_{\inf}(\psi)$ and the set of points with non-decreasing partial quotients, i.e.,
\[
\Lambda=\bigl\{x\in(0,1): a_{n}(x)\leq a_{n+1}(x), \forall n\geq1\bigr\}.
\]
Let us point out that $\dim_{\rm H}\Lambda=\frac{1}{2}$, which is essentially a result of Ramharter \cite{R85}, see Jordan and Rams \cite{lesJR12} for general results in the setting of infinite iterated function systems. In a previous paper \cite{FMSW21}, the authors studied the dimension of the set
$$E_{\inf}(\Lambda,\psi)= E_{\inf}(\psi) \cap\Lambda,$$
and established the following:
    \begin{theorem}\label{inf}
Let $\psi:\mathbb{N}\rightarrow\mathbb{R}^+$ be a function satisfying $\psi(n)\to\infty$ as $n\to\infty$.
\begin{enumerate}[(i)]
\item If $\psi(n)/\log n\to \alpha\ (0\leq\alpha<\infty)$ as $n\to\infty$, then
\[
\dim_{\rm H}E_{\inf}(\Lambda,\psi)=
\left\{
  \begin{array}{ll}
    0, & \hbox{$0\leq\alpha<1$;} \\
    \frac{\alpha-1}{2\alpha}, & \hbox{$\alpha \geq 1$.}
  \end{array}
\right.
\]
\item If $\psi(n)/\log n\to\infty$ as $n\to\infty$, then
\[\dim_{\rm H}E_{\inf}(\Lambda,\psi)=\frac{1}{B+1},\]
where $B$ is given by \eqref{Binf}.
\end{enumerate}
\end{theorem}

In this paper, we investigate the Hausdorff dimension of
$$E_{\sup}(\Lambda,\psi):= E_{\sup}(\psi) \cap\Lambda\quad\text{and}\quad E(\Lambda,\psi):= E(\psi) \cap\Lambda.$$
Our main result is as follows.
\begin{theorem}\label{suplim}
Let $\psi:\mathbb{N}\rightarrow\mathbb{R}^+$ be a function satisfying $\psi(n)\to\infty$ as $n\to\infty$.
\begin{enumerate}[(i)]
\item If $\psi(n)/\log n\to \alpha\ (0\leq\alpha<\infty)$ as $n\to\infty$, then
\begin{equation*}
\dim_{\rm H}E_{\sup}(\Lambda,\psi)= \dim_{\rm H}E(\Lambda,\psi)=
\left\{
  \begin{array}{ll}
    0, & \hbox{$0\leq\alpha<1$;} \\
    \frac{\alpha-1}{2\alpha}, & \hbox{$\alpha \geq 1$.}
  \end{array}
\right.
\end{equation*}
\item If $\psi(n)/\log n\to\infty$ as $n\to\infty$, then
\[\dim_{\rm H}E_{\sup}(\Lambda,\psi)=\frac{1}{A+1},\]
where $A$ is given by \eqref{Asup}, and either $E(\Lambda,\psi)=\emptyset$ or
\[
\dim_{\rm H}E(\Lambda,\psi)=\frac{1}{C+1},
\]
where $C$ is given by \eqref{Clim}.
\end{enumerate}
\end{theorem}

\noindent In the case of $\psi(n)=\alpha \log n$ with $0<\alpha <\infty$, by comparing Theorems \ref{inf} and \ref{suplim} with Theorem \ref{RAMA}, we observe that
$$\dim_{\rm H}E_{\sup}(\psi)=1,\quad \text{and}\quad  \dim_{\rm H}E_{\inf}(\psi)=\dim_{\rm H}E(\psi)=1/2;$$
while the sets $E_{\sup}(\Lambda,\psi)$, $E_{\inf}(\Lambda,\psi)$ and $E(\Lambda,\psi)$ display a multifractal phenomenon.

\indent The paper is organized as follows. In Section 2, we first present some elementary properties and then collect some useful lemmas for computing the Hausdorff dimension of some sets in continued fractions. Section 3 is devoted to the proofs of main results.

\section{Preliminaries}

\subsection{Elementary properties of continued fractions}
 For any $n\geq1$ and $(a_1,\cdots,a_n)\in\mathbb{N}^{n}$, we call
\begin{equation*}
I_{n}(a_1, \cdots, a_n): =\big\{x\in(0,1):\ a_1(x)=a_1, \cdots, a_n(x)=a_n\big\}
\end{equation*}
a \emph{cylinder} of order $n$ associated to $(a_1,\cdots,a_n)$. Denote the $n$-th convergent of the continued fraction expansion of $x$ by
\begin{equation}\label{xnwb}
\frac{p_n(x)}{q_n(x)}:=[a_1(x),a_2(x),\cdots,a_n(x)],
\end{equation}
where $p_n(x)$ and $q_n(x)$ are positive and coprime.
Notice that all points in $I_{n}(a_1, \cdots, a_n)$ have the same $p_n(x)$ and $q_n(x)$. Thus we write
\begin{equation*}
p_n(a_1,\cdots,a_n)=p_n=p_n(x)\ \text{and}\ q_n(a_1,\cdots,a_n)=q_n=q_n(x)
\end{equation*}
for $x\in I_{n}(a_1, \cdots, a_n)$.
It is well known (see \cite[p. 4]{Khi64}) that $p_n$ and $q_n$ satisfy the following recursive formula:
\begin{equation}\label{ppqq}
\begin{cases}
p_n=a_np_{n-1}+p_{n-2};\cr
q_n=a_nq_{n-1}+q_{n-2},
\end{cases}
\end{equation}
with the conventions $p_{-1}\equiv1, p_0\equiv0$ and $q_{-1}\equiv0, q_0\equiv1$. Consequently, $q_n\geq q_{n-1}+q_{n-2}$, and so
\begin{equation}\label{fn}
q_n\geq\frac{1}{\sqrt{5}}\left(\frac{1+\sqrt{5}}{2}\right)^{n}-
\frac{1}{\sqrt{5}}\left(\frac{1-\sqrt{5}}{2}\right)^{n}\geq\frac{1}{2\sqrt{5}}\left(\frac{1+\sqrt{5}}{2}\right)^{n}.
\end{equation}

\begin{prop}[{\cite[p. 18]{IK02}}]\label{cd}
For any $(a_1,\cdots, a_n)\in\mathbb{N}^{n}$, the cylinder $I_{n}(a_1,\cdots, a_n)$ is the interval with the endpoints
$p_n/q_n$ and $(p_n+p_{n-1})/(q_n+q_{n-1})$. As a result, the length of $I_{n}(a_1, \cdots, a_n)$ equals to
\begin{equation*}
|I_{n}(a_1, \cdots, a_n)|=\frac{1}{q_n(q_n+q_{n-1})}.
\end{equation*}
\end{prop}
Combining the second of formula \eqref{ppqq}, \eqref{fn} and Proposition \ref{cd}, we deduce that
 \begin{equation}\label{2cd}
 |I_{n}(a_1, \cdots, a_n)| \leq\frac{1}{q^{2}_n}
 \leq20\left(\frac{1+\sqrt{5}}{2}\right)^{-2n}
 \end{equation}
and
\begin{equation}\label{length}
\left(2^n\prod^n_{k=1}a_k\right)^{-2} \leq |I_n(a_1, \cdots, a_n)| \leq \left(\prod^n_{k=1}a_k\right)^{-2}.
\end{equation}

\subsection{Some useful lemmas}
The first lemma is a combinatorial formula on the cardinality (i.e., $\sharp$) of finite sets.
\begin{lemma}[{\cite[Lemma 4.3]{FWS18}}]\label{card}
For positive integers $\ell$ and $n$, let
\[
D(\ell,n):=\big\{(a_1,\cdots,a_n)\in\mathbb{N}^{n}: 1\leq a_1\leq\cdots \leq a_n\leq\ell\big\}.
\]
Then
 \[\sharp D(\ell,n)=\frac{(n+\ell-1)!}{n!(\ell-1)!}.\]
\end{lemma}

The second lemma established in \cite{FMSW21} provides a method to obtain a lower bound of the Hausdorff dimension of some sets of continued fractions for which their partial quotients are non-decreasing, see Liao and Rams \cite[Lemma 2.3]{LR21} for general results.

\begin{lemma}[{\cite[Lemma 3.4]{FMSW21}}]\label{fmsw}
Let $\{t_n\}$ be a sequence of positive integers tending to infinity with $t_n\geq2$ for all $n\geq1$. Write
\[
\mathbb{E}(\{t_n\}):=\big\{x\in(0,1): nt_n\leq a_n(x)<(n+1)t_n, \forall\,n\geq1\big\}.
\]
Then
\[
\dim_{\rm H}\mathbb{E}(\{t_n\})=\frac{1}{2+\xi},
\]
where $\xi \in [0,\infty]$ is given by
\[
\xi:=\limsup\limits_{n\to\infty}\frac{2\log(n+1)!+\log t_{n+1}}{\log(t_1 t_2\cdots t_n)}.
\]
\end{lemma}

Inspired by the proof of Theorem 2.4 of \cite{FMSW21}, we are able to obtain the following lemma for providing an upper bound of the Hausdorff dimension of the sets $E_{\sup}(\Lambda,\psi)$, $E_{\inf}(\Lambda,\psi)$ and $E(\Lambda,\psi)$.

\begin{lemma}\label{sj}
Let $\alpha_1,\alpha_2$ be two positive real numbers and let
\[\mathbb{E}(\Lambda,\alpha_1,\alpha_2):=\left\{x\in\Lambda: \alpha_1\leq\liminf\limits_{n\to\infty}\frac{\log a_n(x)}{\log n}\leq\alpha_2\right\}.\]
Then for any $\alpha_2\geq\alpha_1\geq1$, we have
\[\dim_{\rm H}\mathbb{E}(\Lambda,\alpha_1,\alpha_2)\leq\frac{\alpha_2-1}{2\alpha_1}.\]
\end{lemma}

\begin{proof}
Let $\alpha_2\geq\alpha_1\geq1$ and $0<\varepsilon<\alpha_1$. For $x\in\mathbb{E}(\Lambda,\alpha_1,\alpha_2)$, there exists $N\geq1$ such that $a_j(x)\geq j^{\alpha_1-\varepsilon}$ for all $j\geq N$; and $a_k(x)\leq k^{\alpha_2+\varepsilon}$ for infinitely many $k$'s. Then
\begin{align*}
\mathbb{E}(\Lambda,\alpha_1,\alpha_2) \subseteq  \bigcup\limits_{N=1}^{\infty} B_{N}(\alpha_1,\alpha_2,\varepsilon),
\end{align*}
where $B_{N}(\alpha_1,\alpha_2,\varepsilon)$ is defined as
\begin{equation}\label{bhgx}
 B_{N}(\alpha_1,\alpha_2,\varepsilon):=\bigcap_{n=N}^{\infty}\bigcup_{k=n}^{\infty}\Big\{x\in\Lambda: a_k(x)\leq k^{\alpha_2+\varepsilon}, a_{j}(x)\geq j^{\alpha_1-\varepsilon}, \forall N\leq j\leq k\Big\}.
 \end{equation}
It follows that
 \begin{equation*}\label{wsgx1}
 \dim_{\rm H}\mathbb{E}(\Lambda,\alpha_1,\alpha_2)\leq\sup_{N\geq1}\big\{\dim_{\rm H}B_{N}(\alpha_1,\alpha_2,\varepsilon)\big\}.
 \end{equation*}
We shall only computer the upper bound of Hausdorff dimension of $B_{1}(\alpha_1,\alpha_2,\varepsilon)$ since the proofs are similar for other cases $N \geq 2$ .
Write
\[
\mathcal{A}_k:=\big\{(a_1, \cdots, a_k)\in\mathbb{N}^{k}:1\leq a_1\leq\cdots \leq a_k \leq k^{\alpha_2+\varepsilon}, a_j\geq j^{\alpha_1-\varepsilon}, \forall 1\leq j\leq k\big\}.
\]
By \eqref{bhgx}, we have
\begin{equation*}
B_{1}(\alpha_1,\alpha_2,\varepsilon)=\bigcap_{n=1}^{\infty}\bigcup_{k=n}^{\infty}\bigcup_{(a_1, \cdots, a_k) \in \mathcal{A}_k} I_k(a_1, \cdots, a_k),
\end{equation*}
which means that for any $n \geq 1$, the family $\{I_k(a_1, \cdots, a_k): k \geq n, (a_1, \cdots, a_k) \in \mathcal{A}_k\}$ is a cover of $B_{1}(\alpha_1,\alpha_2,\varepsilon)$. To estimate the upper bound of $\dim_{\rm H}B_{1}(\alpha_1,\alpha_2,\varepsilon)$, we need the information about the cardinality of $\mathcal{A}_k$ and the length of $I_k(a_1, \cdots, a_k)$.

For the cardinality of $\mathcal{A}_k$, applying Lemma \ref{card}, we deduce that
\begin{align}\label{deltags}
\sharp\mathcal{A}_k \leq \frac{(k+\lfloor k^{\alpha_2+\varepsilon}\rfloor-1)!}{k!(\lfloor k^{\alpha_2+\varepsilon}\rfloor-1)!}&\leq
\frac{\lfloor k^{\alpha_2+\varepsilon}\rfloor\cdot \left(\lfloor k^{\alpha_2+\varepsilon}\rfloor+1\right)\cdots\left(\lfloor k^{\alpha_2+\varepsilon}\rfloor+k-1\right)}{k!} \nonumber\\
&\leq\frac{k^{k(\alpha_2+\varepsilon)}}{k!}\cdot\left(1+\frac{1}{k^{\alpha_2+\varepsilon}}\right)\cdots\left(1+\frac{k-1}{k^{\alpha_2+\varepsilon}}\right)\nonumber\\
&\leq \frac{2^k\cdot k^{k(\alpha_2+\varepsilon)}}{k!}.
\end{align}
Recall the Stirling formula, we get that
\begin{equation*}\label{st}
\sqrt{2\pi}k^{k+\frac{1}{2}}\exp(-k)\leq k!\leq\exp(1)k^{k+\frac{1}{2}}\exp(-k),
\end{equation*}
which, in combination with \eqref{deltags}, implies that
\begin{equation}\label{jqgs}
 \sharp\mathcal{A}_k \leq 2^{k}\cdot (k!)^{\alpha_2+\varepsilon-1}\cdot\left(\frac{\exp(k)}{\sqrt{2\pi k}}\right)^{\alpha_2+\varepsilon}\leq 2^{k}\cdot\exp((\alpha_2+\varepsilon)k) \cdot(k!)^{\alpha_2+\varepsilon-1}.
\end{equation}
For the length of $I_k(a_1, \cdots, a_k)$, it follows from \eqref{length} that
\begin{equation}\label{zzss}
|I_k(a_1,\cdots, a_k)|   \leq \left(\prod\limits_{j=1}^{k}a _{j}\right)^{-2} \leq (k!)^{-2(\alpha_1-\varepsilon)}.
\end{equation}
Let $s:=\frac{\alpha_2+2\varepsilon-1}{2(\alpha_1-\varepsilon)}$. Denote by $\mathcal{H}^{s}$ the $s$-dimensional Hausdorff measure. We conclude from \eqref{jqgs} and \eqref{zzss} that
\begin{align*}
\mathcal{H}^{s}(B_{1}(\alpha_1,\alpha_2,\varepsilon))& \leq\liminf_{n\to\infty}\sum\limits_{k=n}^{\infty}\sum\limits_{(a_1, \cdots, a_k) \in \mathcal{A}_k}|I_{k}(a_1, \cdots,a_k)|^{s}\\
& \leq \liminf_{n\to\infty}\sum\limits_{k=n}^{\infty} \sharp \mathcal{A}_k \cdot (k!)^{-2s(\alpha_1-\varepsilon)}\\
&\leq\liminf_{n\to\infty}\sum\limits_{k=n}^{\infty}\frac{2^{k}\cdot\exp((\alpha_2+\varepsilon)k) \cdot(k!)^{\alpha_2+\varepsilon-1}}{(k!)^{\alpha_2+2\varepsilon-1}}\\
&=\liminf_{n\to\infty}\sum\limits_{k=n}^{\infty}\frac{2^{k}\cdot\exp((\alpha_2+\varepsilon)k)}
{(k!)^{\varepsilon}}=0.
\end{align*}
This shows that
\[\dim_{\rm H}B_{1}(\alpha_1,\alpha_2,\varepsilon)\leq\frac{\alpha_2+2\varepsilon-1}{2(\alpha_1-\varepsilon)}.\]
Letting $\varepsilon \to 0^+$, we obtain the desired upper bound.
\end{proof}

\section{Proofs of main results}
In this section, we will prove Theorem \ref{suplim}. The proof is divided into two cases: the Hausdorff dimension of $E_{\sup}(\Lambda,\psi)$ and the Hausdorff dimension of $E(\Lambda,\psi)$.

\subsection{Hausdorff dimension of $E_{\sup}(\Lambda,\psi)$}
Recall that
\[
E_{\sup}(\Lambda,\psi)=\left\{x\in\Lambda:\ \limsup\limits_{n\to\infty}\frac{\log a_n(x)}{\psi(n)}=1\right\}.
\]
We will give the proof of Theorem \ref{suplim} for the Hausdorff dimension of $E_{\sup}(\Lambda,\psi)$ when $\psi(n)/\log n\to\alpha\ (0\leq\alpha<\infty)$ and $\psi(n)/\log n\to \infty$ respectively.

\subsubsection{Case $\psi(n)/\log n\to\alpha\ (0\leq\alpha<\infty)$}
For the upper bound of $\dim_{\rm H}E_{\sup}(\Lambda,\psi)$, we remark that
\begin{equation}\label{inf45}
E_{\sup}(\Lambda,\psi) \subseteq\left\{x\in\Lambda: \liminf\limits_{n\to\infty}\frac{\log a_n(x)}{\log n}\leq\alpha\right\}.
\end{equation}
So it is sufficient to give the upper bound of the Hausdorff dimension of the set on the right-hand side of \eqref{inf45}.

\begin{lemma}\label{infles}
For $0\leq\alpha <\infty$,
\[
\dim_{\rm H}\left\{x\in\Lambda: \liminf\limits_{n\to\infty}\frac{\log a_n(x)}{\log n}\leq\alpha\right\} \leq
\left\{
  \begin{array}{ll}
    0, & \hbox{$0\leq\alpha<1$;} \\
    \frac{\alpha-1}{2\alpha}, & \hbox{$\alpha \geq 1$.}
  \end{array}
\right.
\]
\end{lemma}

\begin{proof}
For $0\leq\alpha <1$, let $0<\varepsilon<1-\alpha$. By the definition of liminf,
\begin{equation}\label{1bhgx}
\left\{x\in\Lambda: \liminf\limits_{n\to\infty}\frac{\log a_n(x)}{\log n}\leq\alpha\right\}\subseteq\bigcap_{n=1}^{\infty}\bigcup_{k=n}^{\infty}\bigcup_{(a_1, \cdots, a_k)\in  \mathcal{C}_k}I_k(a_1, \cdots, a_k),
\end{equation}
where $\mathcal{C}_{k}$ is given by
\[
\mathcal{C}_{k}:=\big\{(a_1, \cdots, a_k)\in\mathbb{N}^{k}: 1\leq a_1\leq \cdots\leq a_k\leq k^{\alpha+\varepsilon}\big\}.
\]
Note that the cardinality of $\mathcal{C}_k$ satisfies
\begin{align}\label{ak33}
\nonumber \sharp \mathcal{C}_k=\frac{(k+\lfloor k^{\alpha+\varepsilon}\rfloor-1)!}{k!(\lfloor k^{\alpha+\varepsilon}\rfloor-1)!}
&\leq (k+1)\cdots(k+\lfloor k^{\alpha+\varepsilon}\rfloor-1)\\
&<(k+k^{\alpha+\varepsilon})^{k^{\alpha+\varepsilon}} \nonumber\\
&< \exp(k^{\alpha+\varepsilon}(\log k+1)),
\end{align}
and for any $(a_1, \cdots, a_k) \in \mathcal{C}_k$, it derives from \eqref{2cd} that
\begin{equation}\label{20}
|I_{k}(a_1, \cdots, a_k)|  \leq20\left(\frac{1+\sqrt{5}}{2}\right)^{-2k}.
\end{equation}
Taking $s=\varepsilon$ and combining \eqref{1bhgx}, \eqref{ak33} and \eqref{20}, we conclude that the $s$-dimensional Hausdorff measure of the set on the right-hand side of \eqref{inf45} is not greater than
\begin{align*}
\liminf_{n\to\infty}\sum\limits_{k=n}^{\infty}\sum\limits_{(a_1, \cdots, a_k) \in \mathcal{C}_k}|I_{k}(a_1, \cdots,a_k)|^{s}
& \leq \liminf_{n\to\infty}\sum\limits_{k=n}^{\infty} \sharp \mathcal{C}_k \cdot \frac{20^{\varepsilon}}{\big((1+\sqrt{5})/2\big)^{2k\varepsilon}}\\
& \leq \liminf_{n\to\infty}\sum\limits_{k=n}^{\infty}\frac{20^{\varepsilon}\cdot\exp(k^{\alpha+\varepsilon}(\log k+1))}{\big((1+\sqrt{5})/2\big)^{2k\varepsilon}}=0,
\end{align*}
which yields that
\[
\dim_{\rm H}\left\{x\in\Lambda: \liminf\limits_{n\to\infty}\frac{\log a_n(x)}{\log n}\leq\alpha\right\} \leq 0
\]
since $\varepsilon$ is arbitrary.

For $\alpha \geq 1$, we point out that
\begin{equation}\label{ii}
\left\{x\in\Lambda: \liminf\limits_{n\to\infty}\frac{\log a_n(x)}{\log n}\leq\alpha\right\} = \left\{x\in\Lambda: \liminf\limits_{n\to\infty}\frac{\log a_n(x)}{\log n}<1\right\} \bigcup \mathbb{E}(\Lambda,1,\alpha),
\end{equation}
where $\mathbb{E}(\Lambda,1,\alpha)$ is defined as in Lemma \ref{sj}. Note that
\[
\left\{x\in\Lambda: \liminf\limits_{n\to\infty}\frac{\log a_n(x)}{\log n}<1\right\} = \bigcup^\infty_{K=1} \left\{x\in\Lambda: \liminf\limits_{n\to\infty}\frac{\log a_n(x)}{\log n} \leq 1 -\frac{1}{K}\right\},
\]
so it has Hausdorff dimension zero. By \eqref{ii}, we need only consider the upper bound of $\dim_{\rm H}\mathbb{E}(\Lambda,1,\alpha)$. When $\alpha =1$, we have $\dim_{\rm H}\mathbb{E}(\Lambda,1,\alpha) =0$ by Theorem \ref{inf}. When $\alpha >1$, since for any $n >\alpha-1$,
\[
\mathbb{E}(\Lambda,1,\alpha) \subseteq \bigcup^{n-1}_{k=0} \mathbb{E}(\Lambda,1+k\cdot\frac{\alpha-1}{n}, 1+(k+1)\cdot\frac{\alpha-1}{n}),
\]
it follows from Lemma \ref{sj} that
\begin{align*}
\dim_{\rm H}\mathbb{E}(\Lambda,1,\alpha) &\leq \max_{0 \leq k \leq n-1} \left\{\dim_{\rm H}\mathbb{E} (\Lambda,1+k\cdot\frac{\alpha-1}{n}, 1+(k+1)\cdot\frac{\alpha-1}{n})\right\} \nonumber \\
&\leq \max_{0 \leq k \leq n-1} \left\{ \frac{(k+1)\cdot \frac{\alpha-1}{n}}{2(1+k\cdot \frac{\alpha-1}{n})}\right\}.
\end{align*}
Combining this with the fact that for $0<\beta <1$, the map $x\mapsto \frac{(x+1)\beta}{2(1+x\beta)}$ is increasing, we see that
\begin{align*}
\dim_{\rm H}\mathbb{E}(\Lambda,1,\alpha)  \leq  \frac{\alpha-1}{2(1+(\alpha-1)\cdot\frac{n-1}{n})}.
\end{align*}
Letting $n \to \infty$, we get that $\dim_{\rm H}\mathbb{E}(\Lambda,1,\alpha)  \leq (\alpha -1)/(2\alpha)$. In view of \eqref{ii},
\begin{equation*}
\dim_{\rm H}\left\{x\in\Lambda: \liminf\limits_{n\to\infty}\frac{\log a_n(x)}{\log n}\leq\alpha\right\}  \leq \frac{\alpha -1}{2\alpha}.
\end{equation*}
\end{proof}

For the lower bound of $\dim_{\rm H}E_{\sup}(\Lambda,\psi)$, when $0 \leq \alpha \leq 1$, we have $\dim_{\rm H}E_{\sup}(\Lambda,\psi) =0$; when $\alpha >1$, let $t_n:=2\lfloor n^{\alpha-1}\rfloor$ and
\[
\mathbb{E}(\{t_n\})=\big\{x\in (0,1): nt_n\leq a_n(x)<(n+1)t_n, \forall n\geq1\big\}.
\]
Then $\{t_n\}$ is non-decreasing, and so $\mathbb{E}(\{t_n\})$ is a subset of $E_{\sup}(\Lambda,\psi)$. Since
\[
\xi= \limsup\limits_{n\to\infty}\frac{2\log(n+1)!+\log t_{n+1}}{\log(t_1 t_2\cdots t_n)} = \frac{2}{\alpha -1},
\]
applying Lemma \ref{fmsw}, we deduce that
\begin{align*}
\dim_{\rm H}E_{\sup}(\Lambda,\psi)\geq\dim_{\rm H}\mathbb{E}(\{t_n\}) = \frac{1}{2+\xi} = \frac{\alpha -1}{2\alpha}.
\end{align*}

\subsubsection{Case $\psi(n)/\log n\to\infty$ }

For the upper bound of $\dim_{\rm H}E_{\sup}(\Lambda,\psi)$, we remark that $E_{\sup}(\Lambda,\psi) \subseteq F(\Lambda,\psi)$, where $F(\Lambda,\psi)$ is given by
\[
F(\Lambda,\psi):= \Big\{x\in\Lambda: a_n(x)\geq 2^{\psi(n)}\ \text{for infinitely many $n$'s}\Big\}.
\]

\begin{lemma}\label{varphi}
Let $\varphi: \mathbb{N} \to \mathbb{R}^+$ be a function. Then
\[
\dim_{\rm H}F(\Lambda,\varphi) = \frac{1}{\gamma+1},
\]
where $\gamma \in [1,\infty]$ is given by
\[
\log \gamma:=\liminf_{n \to \infty} \frac{\log\varphi(n)}{n}.
\]
\end{lemma}

\begin{proof}
For $a, b>1$, let
\[
F(\Lambda,a,b):= \Big\{x\in\Lambda: a_n(x)\geq a^{b^n}, \forall n \geq 1\Big\}
\]
and
\[
\widetilde{F}(\Lambda,a,b):= \Big\{x\in\Lambda: a_n(x)\geq a^{b^n}\ \text{for infinitely many $n$'s}\Big\}.
\]
We claim that
\begin{equation*}\label{Lab}
\dim_{\rm H} F(\Lambda,a,b) = \dim_{\rm H} \widetilde{F}(\Lambda,a,b) = \frac{1}{b+1}.
\end{equation*}
In fact, the lower bound of $\dim_{\rm H} F(\Lambda,a,b)$ can be read off from Lemma \ref{fmsw} by putting $t_n:=2a^{b^n}$ and the upper bound of $\dim_{\rm H} \widetilde{F}(\Lambda,a,b)$ follows from the result of {\L}uczak \cite{Luc97} (see \cite{WW08} for general results).
Next we are ready to deal with $\dim_{\rm H}F(\Lambda,\varphi)$ according to $\gamma =1$, $1<\gamma <\infty$ and $\gamma =\infty$ respectively.

When $\gamma =1$, for any small $\varepsilon >0$, we see that $\varphi(n) \leq (1+\varepsilon)^n$ for infinitely many $n$'s. Then $F(\Lambda,2, 1+\varepsilon) \subseteq F(\Lambda,\varphi)$, and so
\[
\frac{1}{2+\varepsilon} = \dim_{\rm H}F(\Lambda,2, 1+\varepsilon) \leq \dim_{\rm H} F(\Lambda,\varphi) \leq \dim_{\rm H} \Lambda = \frac{1}{2}.
\]
Letting $\varepsilon \to 0^+$, we obtain $\dim_{\rm H} F(\Lambda,\varphi) = 1/2 =1/(\gamma+1)$.

When $1<\gamma <\infty$, for any small $0<\varepsilon <\gamma-1$, we see that $\varphi(n) \geq (\gamma-\varepsilon)^n$ for sufficiently large $n$, and $\varphi(n) \leq (\gamma+\varepsilon)^n$ for infinitely many $n$'s.
Then $F(\Lambda,2, \gamma+\varepsilon) \subseteq F(\Lambda,\varphi) \subseteq \widetilde{F}(\Lambda,2,\gamma-\varepsilon)$, and so
\[
\frac{1}{\gamma+\varepsilon +1} =\dim_{\rm H}F(\Lambda,2, \gamma+\varepsilon) \leq \dim_{\rm H} F(\Lambda,\varphi) \leq \dim_{\rm H}\widetilde{F}(\Lambda,2,\gamma-\varepsilon) = \frac{1}{\gamma-\varepsilon +1}.
\]
Since $\varepsilon$ is arbitrary, we have $\dim_{\rm H} F(\Lambda,\varphi) =1/(\gamma+1)$.

When $\gamma =\infty$, for any large $K>1$, we see that $\varphi(n) \geq K^n$ for sufficiently large $n$. Then $F(\Lambda,\varphi) \subseteq \widetilde{F}(\Lambda,2,K)$, and so
\[
\dim_{\rm H} F(\Lambda,\varphi) \leq \dim_{\rm H}\widetilde{F}(\Lambda,2,K) = \frac{1}{K +1},
\]
which implies that $\dim_{\rm H} F(\Lambda,\varphi) = 0 =1/(\gamma+1)$ by letting $K \to \infty$.
\end{proof}

From Lemma \ref{varphi}, we deduce that
\[
\dim_{\rm H}E_{\sup}(\Lambda,\psi) \leq \dim_{\rm H} F(\Lambda,\psi) = \frac{1}{A+1}\ \ \text{with}\ \ \log A =\liminf_{n \to \infty} \frac{\log\psi(n)}{n}.
\]

For the lower bound of $\dim_{\rm H}E_{\sup}(\Lambda,\psi)$, we shall construct a suitable subset of $E_{\sup}(\Lambda,\psi)$. To this end, we follow the notation used in \cite[p.\,901--903]{FMS20}.
Let
\begin{equation*}
\theta(n):=\min\limits_{k\geq n}\{\psi(k)\},\ \ \forall n\geq1
\end{equation*}
and define a sequence $\{d_n\}$ as follows:
\begin{equation*}
d_1:=\exp(\theta(1))\ \ \text{and}\ \ d_n:=\min\left\{\exp(\theta(n)),\prod\limits_{k=1}^{n-1}d^{A-1+\varepsilon}_k\right\}\ (n\geq2).
\end{equation*}
Then $d_{n+1}\geq d_n\,(\forall n\geq2)$,
\begin{align}\label{dnlogn}
\lim_{n\to\infty}\frac{\log d_{n}}{\log n} =\infty,
\end{align}

\begin{align}\label{wan3}
\limsup\limits_{n\to\infty}\frac{\log d_{n+2}}{\log d_2+\cdots+\log d_{n+1}} \leq A-1+\varepsilon
\end{align}
and
\begin{equation}\label{cn}
\limsup_{n \to \infty} \frac{\log d_n}{\psi(n)}=1.
\end{equation}
Let $t_n:= 2d_{n+1}$ for all $n \geq 1$. Then $\{t_n\}$ is non-decreasing. Write
\[
\mathbb{E}(\{t_n\})=\big\{x\in (0,1): nt_n\leq a_n(x)<(n+1)t_n, \forall n\geq1\big\}.
\]
By \eqref{cn} and the condition $\psi(n)/\log n\to\infty$ as $n \to \infty$, we see that $\mathbb{E}(\{t_n\})$ is a subset of $E_{\sup}(\Lambda,\psi)$. It follows from Lemma \ref{fmsw} that
\begin{equation*}
\dim_{\rm H}E_{\sup}(\Lambda,\psi) \geq \dim_{\rm H}\mathbb{E}(\{t_n\})= \frac{1}{2+\xi}\ \ \text{with}\ \ \xi= \limsup\limits_{n\to\infty}\frac{2\log(n+1)!+
\log d_{n+2}}{\log d_2+\cdots+\log d_{n+1}}.
\end{equation*}
By the Stolz-Ces\`{a}ro theorem, \eqref{dnlogn} and \eqref{wan3}, we deduce that
\begin{align*}
\xi &\leq \limsup\limits_{n\to\infty}\frac{2\log(n+1)!}{\log d_2+\cdots+\log d_{n+1}}
+\limsup\limits_{n\to\infty}\frac{\log d_{n+2}}{\log d_2+\cdots+\log d_{n+1}}\\
\nonumber &\leq \limsup\limits_{n\to\infty}\frac{2\log(n+1)}{\log d_{n+1}}+
A-1+\varepsilon\\
&=A-1+\varepsilon.
\end{align*}
Therefore,
\begin{align*}
\dim_{\rm H}E_{\sup}(\Lambda,\psi)\geq \frac{1}{A+1+\varepsilon}.
\end{align*}
Since $\varepsilon$ is arbitrarily, we get that $\dim_{\rm H}E_{\sup}(\Lambda,\psi)\geq 1/(A+1)$.

\subsection{Hausdorff dimension of $E(\Lambda, \psi)$}

Let $\psi$ and $\widetilde{\psi}$ be positive functions defined on $\mathbb{N}$. We say that $\psi$ and $\widetilde{\psi}$ are \emph{equivalent} if $\psi(n)/\widetilde{\psi}(n)\to1$ as $n\to\infty$.
Recall that
\[
E(\Lambda, \psi)=\left\{x\in \Lambda: \lim\limits_{n\to\infty}\frac{\log a_n(x)}{\psi(n)}=1\right\}.
\]

\begin{lemma}\label{limnon-empty}
$E(\Lambda,\psi)\neq\emptyset$ if and only if $\psi$ is equivalent to a non-decreasing function.
\end{lemma}
\begin{proof}
If $E(\Lambda,\psi)\neq\emptyset$, then we take $x_0\in E(\Lambda,\psi)$, and so
\[
a_{n+1}(x_0)\geq a_n(x_0), \forall n \geq 1\ \ \text{and}\ \ \lim\limits_{n\to\infty}\frac{\log a_n(x_0)}{\psi(n)}=1.
\]
Define $\widetilde{\psi}(n):=\lfloor\log a_n(x_0)\rfloor+1$ for all $n\geq1$. Then we see that $\widetilde{\psi}$ is non-decreasing and is equivalent to $\psi$.

Suppose that $\psi$ and $\widetilde{\psi}$ are equivalent and $\widetilde{\psi}$ is non-decreasing. Define a point $\widetilde{x}\in(0,1)$ such that $a_n(\widetilde{x})=\lfloor\exp(\widetilde{\psi}(n))\rfloor$ for all $n\geq1$.
Then $a_{n+1}(\widetilde{x})\geq a_n(\widetilde{x}), \forall n \geq 1$ and
\[
\lim\limits_{n\to\infty}\frac{\log a_n(\widetilde{x})}{\psi(n)}= \lim\limits_{n\to\infty}\frac{\widetilde{\psi}(n)}{\psi(n)}=1.
\]
That is to say, $\widetilde{x}\in E(\Lambda,\psi)$, and thus $E(\Lambda,\psi)\neq\emptyset$.
\end{proof}

We remark that $E(\Lambda,\psi)=E(\Lambda,\tilde{\psi})$ if $\psi$ and $\tilde{\psi}$ are equivalent. By Lemma \ref{limnon-empty}, we assume that $\psi$ is non-decreasing in dealing with $E(\Lambda, \psi)$.

\subsubsection{Case $\psi(n)/\log n\to\alpha\ (0\leq\alpha<\infty)$}

For the upper bound of $\dim_{\rm H}E(\Lambda, \psi)$, we see that
\[
E(\Lambda, \psi) \subseteq \left\{x\in \Lambda: \lim\limits_{n\to\infty}\frac{\log a_n(x)}{\log n}=\alpha\right\}\subseteq \left\{x\in \Lambda: \liminf\limits_{n\to\infty}\frac{\log a_n(x)}{\log n}=\alpha\right\}.
\]
It follows from Theorem \ref{inf} that
\[
\dim_{\rm H}E(\Lambda, \psi) \leq
\left\{
  \begin{array}{ll}
    0, & \hbox{$0\leq\alpha<1$;} \\
    \frac{\alpha-1}{2\alpha}, & \hbox{$\alpha \geq 1$.}
  \end{array}
\right.
\]

For the lower bound of $\dim_{\rm H}E(\Lambda, \psi)$, when $0 \leq \alpha \leq 1$, we have $\dim_{\rm H}E(\Lambda,\psi) =0$; when $\alpha >1$, let $t_n:=2\lfloor n^{\alpha-1}\rfloor$ and
\[
\mathbb{E}(\{t_n\})=\big\{x\in (0,1): nt_n\leq a_n(x)<(n+1)t_n, \forall n\geq1\big\}.
\]
Then $\{t_n\}$ is non-decreasing, and so $\mathbb{E}(\{t_n\})$ is a subset of $E(\Lambda,\psi)$. Since
\[
\xi= \limsup\limits_{n\to\infty}\frac{2\log(n+1)!+\log t_{n+1}}{\log(t_1 t_2\cdots t_n)} = \frac{2}{\alpha -1},
\]
applying Lemma \ref{fmsw}, we deduce that
\begin{align*}
\dim_{\rm H}E(\Lambda,\psi)\geq\dim_{\rm H}\mathbb{E}(\{t_n\}) = \frac{1}{2+\xi} = \frac{\alpha -1}{2\alpha}.
\end{align*}

\subsubsection{Case $\psi(n)/\log n\to\infty$}

For the upper bound of $\dim_{\rm H}E(\Lambda,\psi)$, it follows from Theorem \ref{RAMA} (v) that
\[
\dim_{\rm H}E(\Lambda,\psi) \leq \dim_{\rm H}E(\psi) = \frac{1}{C+1}.
\]

For the lower bound of $\dim_{\rm H}E(\Lambda,\psi)$, let $t_n:=\lfloor \exp(\psi(n)+1)\rfloor$, then $t_n\geq2$ and $\{t_n\}$ is non-decreasing since $\psi$ is non-decreasing.
Write
\[
\mathbb{E}(\{t_n\})=\Big\{x\in(0,1): nt_n\leq a_n(x)<(n+1)t_n, \forall n\geq1\Big\}.
\]
Since $\psi(n)/\log n\to\infty$ as $n \to \infty$, we deduce that $\mathbb{E}(\{t_n\})\subseteq E(\Lambda,\psi)$. Applying Lemma \ref{fmsw}, we conclude that
 \begin{align*}
 \dim_{\rm H}E(\Lambda,\psi)\geq\dim_{\rm H}\mathbb{E}(\{t_n\}) = \frac{1}{2+\xi} \ \ \text{with}\ \ \xi= \limsup\limits_{n\to\infty}\frac{2\log(n+1)!+
\psi(n+1)}{\psi(1) +\cdots +\psi(n)}.
\end{align*}
By the Stolz-Ces\`{a}ro theorem, we get that
\begin{align*}
\xi &\leq \limsup\limits_{n\to\infty}\frac{2\log(n+1)!}{\psi(1) +\cdots +\psi(n)} + \limsup\limits_{n\to\infty}\frac{\psi(n+1)}{\psi(1) +\cdots +\psi(n)}\\
&\leq \limsup\limits_{n\to\infty}\frac{\log(n+1)}{\psi(n)} + C-1\\
&= C-1.
\end{align*}
Therefore,
\[
\dim_{\rm H}E(\Lambda,\psi)\geq \frac{1}{C+1}.
\]

{\bf Acknowledgement:}
The research is supported by the National Natural Science Foundation of China (Nos.\,11771153, 11801591, 11971195, 12071171, 12171107), Jiangsu Province Innovation \& Entrepreneurship Doctor Talent Program (No.\,JSSCBS20210201) and Guangdong Basic and Applied Basic Research Foundation (No.\,2021A1515010056).

\end{document}